\documentclass[11pt,a4paper]{article}
\usepackage{amsmath,amssymb,amsthm,mathptmx}
\newtheorem{theorem}{Theorem}
\newtheorem{lemma}{Lemma}
\author{Andrey Yu. Rumyantsev}
\title{Everywhere complex sequences and\\ probabilistic
method\thanks{Supported by RFBR 0901-00709a and
NAFIT ANR-08-EMER-008 grants.}}

\date{}

\DeclareMathOperator{\K}{\mathrm{K}}

\def\INT{\mathbb{N}}

\begin{document}
\maketitle

\begin{abstract}
The main subject of the paper is everywhere complex sequences.
An everywhere complex sequence is a sequence that does not
contain substrings of Kolmogorov complexity less than $\alpha
n-O(1)$ where $n$ is the length of substring and $\alpha$ is a
constant between $0$ and~$1$.

First, we prove that no randomized algorithm can produce
everywhere complex sequence with positive probability.

On the other hand, for weaker notions of everywhere complex
sequences the situation is different. For example, there is a
probabilistic algorithm that produces (with probability~$1$)
sequences whose substrings
of length $n$ have complexity $\sqrt{n}-O(1)$.

Finally, one may replace the complexity of a substring (in the
definition of everywhere complex sequence) by its conditional
complexity when the position is given. This gives a stronger
notion of everywhere complex sequence,
and no randomized algorighm can produce (with positive
probability) such a sequence even if $\alpha n$ is replaced by
$\sqrt{n}$, $\log^* n $ or any other monotone unbounded
computable function.
\end{abstract}

\section{Introduction}

The paper considers binary sequences with substrings of high
Kolmogorov complexity. Kolmogorov complexity of a binary string
is the minimal length of a program that produces this string. We
refer the reader to~\cite{livitan} or~\cite{shen-lecture-notes}
for the definition and basic properties of Kolmogorov
complexity.

The Levin--Schnorr Theorem (see, e.g.,~\cite{livitan}) characterizes
randomness of a sequence in terms of complexity of its prefixes.
It implies that a $n$-bit prefix of a Martin-L\"of random sequence
has complexity $n-O(1)$.
(Technically, we should consider monotone or prefix complexity
here; for plain complexity we have $n-O(\log n)$ bound, but
in this paper logarithmic precision is enough.) So
sequences with complex prefixes exists (and, moreover, fair coin
tossing produces such a sequence with probability~$1$).

If we require that all substrings (not only prefixes) are
complex, the situation changes. Random sequences no more have
this property, since every random sequence contains arbitrarily
long groups of consequtive zeros (and these groups have very
small complexity).

However, the sequences with this property (``everywhere
complex'') still exist. The following Lemma (proved by Levin,
see~\cite{dls}) says that there exists a sequence where every
substring has high complexity (though the condition is now
weaker, the complexity is greater than $\alpha n-O(1)$ where $n$
is the length and $0<\alpha<1$).

Here is the exact statement. We denote by $\omega([i,j))$ a
substring $\omega_i\omega_{i+1}\omega_{i+2}\dots\omega_{j-1}$ of
a sequence $\omega$; by $K(u)$ we denote Kolmogorov complexity
of a binary string $u$.

\begin{lemma}[Levin]
\label{levin}
Let $\alpha$ be a real number, $0<\alpha<1$. There exists a
sequence $\omega$ such that
    $$
\K(\omega([k,k+n)))\ge\alpha n-O(1).
    $$
for all natural numbers $k$ and $n$.
\end{lemma}

Here the constant $O(1)$ may depend on $\alpha$ but not on $n$ and $k$.

Levin's proof in~\cite{dls} used complexity arguments:
informally, we construct the sequence from left to right adding
bit blocks; each new block should increase the complexity as
much as possible.

Later it became clear that this lemma has combinatorial meaning:
if for every $n$ some $2^{\alpha n}$ strings of length $n$ are
``forbidden'', there exists an infinite sequence without long
forbidden substrings. This combinatorial interpretaion shows
that the statement of the lemma (and even a stronger statement
about subsequences, not only substrings) is a corollary of the
Lovasz local lemma (see~\cite{rumyantsev-1,rumyantsev-2}).
Recently two more proofs were suggested (by Joseph
Miller~\cite{miller-two-notes} and Andrej Muchnik).

Before stating our results, let us mention the following
slightly generalized version of Levin's lemma. Though not stated
explicitly in~\cite{dls}, it can be proved by the same argument.

\begin{lemma}[Levin, generalized]
\label{genlevin}
Let $\alpha$ be a real number, $0<\alpha<1$. Then there exists
a sequence $\omega$ such that
    $$
\K(\omega([k,k+n))\mid k,n)\ge\alpha n-O(1).
    $$
for all integers $k$, $n$.
\end{lemma}

Here $K(x|y)$ denotes conditional Kolmogorov complexity of a string
$x$ when $y$ is given (i.e., the minimal length of a program
that transforms $y$ to $x$).
The difference is that substrings are now complex with respect
to their position and length (so, for example, the binary
representation of $n$ can not appear starting from
position~$n$). In combinatorial terms, we have different sets of
forbidden substrings for different positions.

One can ask how ``constructive'' could be the proofs of Levin's
lemma and its variants. There are several different versions of
this question. One may assume that the set of forbidden strings
is decidable and ask whether there exists a computable sequence
that avoids all sufficiently long forbidden strings. Miller's
argument shows that this is indeed the case, though a similar
question of 2D configurations (instead of 1D sequences,
cf.~\cite{rumyantsev-1}) is still open.

In this paper we consider a different version of this question
and ask whether there exists a probabilistic algorithm that
produces a sequence satisfying the statement of Levin's Lemma
(or some version of it) with positive probability.

\section{The results}

We say that a sequence $\omega$ is \emph{$\alpha$-everywhere
complex} if
     $$\K(\omega([k,k+n)))\ge\alpha n-c$$
for some constant $c$ and for all integers $k$ and $n$.

\begin{theorem}\label{th1}
No probabilistic algorithm can produce with positive
probability a sequence~$\omega$ that is $\alpha$-everywhere complex
for some $\alpha\in(0,1)$.
\end{theorem}

\begin{theorem}\label{th2}
Let $\sum_{i=0}^{\infty}a_i$ be a computable converging series
of nonnegative rational numbers. There exists a probabilistic
algorithm that produces with probability~$1$ some
sequence~$\omega$ such that
         $$\K(\omega([k,k+n)))\ge a_{[\log n]}n-c$$
for some $c$ and for all $k$ and $n$.
\end{theorem}

\begin{theorem}\label{th3}
No probabilistic algorithm can produce with positive probability
a sequence~$\omega$ with the following property\textup{:} there
exists a non-decreasing unbounded computable function
$g\colon\INT\to\INT$ such that
         $$\K(\omega([k,k+n))\mid k,n)\ge g(n)$$
for all $k$ and $n$.
\end{theorem}

Theorem~\ref{th1} and \ref{th2} complement each other: the first
one says that $\alpha$-everywhere complex sequences for a fixed
$\alpha>0$ (even very small) cannot be obtained by a
probabilistic algorithm; the second one says that if we allow
sublinear growth and replace the bound $\alpha n$ by $\sqrt{n}$
or $n/\log^2 n$, then the probabilistic algorithm exists. (There
are intermediate cases where none of these theorems is
applicable, say, $n/\log
n$ bound; we do not know the answer for these cases.)

Theorem~\ref{th3} says that Theorem~\ref{th2} cannot be extended
to the case of generalized Levin's lemma; here the answer is
negative for any computable non-decreasing unbounded function.

\section{Proof of Theorem~\ref{th2}}

Let us start with the positive result.

\begin{proof}[Proof of theorem~\ref{th2}]
The idea of the construction is simple. We fix some computable
function $f\colon\mathbb{N}\to\mathbb{N}$ and then let
$\omega_i=\tau_{f(i)}$ where $\tau_i$ is a sequence of random
bits (recall that we construct a probabilistic algorithm that
uses random bit generator).

In other words, we repeat the same random bit $\tau_j$ several
times at the locations $\omega_i$ where $f(i)=j$. Why does this
help? It allows us to convert bounds for the complexity of
\emph{prefixes} of $\tau$ into bounds for the complexity of
\emph{substrings} of $\omega$. Indeed, if we have some substring
of $\omega$ and some additional information that tell us where
several first bits of $\tau$ are located in the substring, we
can reconstruct a prefix of $\tau$.

More details. We may assume without loss of generality that $n$,
the length of a substring, is large enough. We may also assume
that $n$ is a power of $2$, i.e., that $n=2^m$ for some $m$.
Indeed, for every substring $x$ we can consider its prefix $x'$
whose length is the maximal power of $2$ not exceeding the
length of $x$. The bound for complexity of $x'$ gives the same
(up to a constant factor) bound for the complexity of $x$.

Consider the substring $\omega([k,k+2^m))$ for some $k$ and $m$.
We want it to contain all the bits from some prefix of $\tau$,
more specifically, the first $a_m 2^m$ bits
$\tau_0,\ldots,\tau_{a_m 2^m-1}$ of $\tau$. (We may assume without loss of
generality that $a_m2^m$ is an integer.)

To achieve this, we put each of these bits at the positions that
form an arithmetic progression with common difference $2^m$. The
first term of this progression will be smaller than its
difference, and therefore each interval of length $2^m$ contains
exactly one term of this progression.

In this way for a given $m$ we occupy $a_m$-fraction of the
entire space of indices (each progression has density $1/2^m$
and we have $a_m 2^m$ of them). So to have enough room for all
$m$ we need that $\sum a_m\le 1$. This may be not the case at
first, but we can start with large enough $m_0$ to make the tail
small.

Technically, first we let $m=m_0$ and split $\mathbb{N}$ into
$2^{m}$ arithmetic progressions with difference $2^{m}$. (First
progression is formed by multiples of $2^{m}$, the second is
formed by numbers that are equal to $1$ modulo $2^{m}$, etc.) We
use first $a_m 2^m$ of them for level $m$ reserving the rest for
higher levels. Then we switch to level $m=m_0+1$, splitting
each remaining progression into two (even and odd terms), use some of them
for level $m_0+1$, convert the rest into progressions with twice
bigger difference for level $m_0+2$, etc. (Note that if in a
progression the first term is less than its difference, the same
is true for its two halves.)

This process continues indefinitely, since we assume that
$a_{m_0}+a_{m_0+1}+\ldots \le 1$. Note that even if this sum is
\emph{strictly} less than $1$, all natural numbers will be
included in some of the progressions: indeed, at each step we
cover the least uncovered yet number. So we have described a
total computable function $f$ (its construction depends on
$m_0$, see below).

Now we translate lower bounds for complexity of prefixes of
$\tau$ into bounds for complexity of substrings of $\omega$: the
substring $\omega([k,k+2^m))$ contains first $a_m 2^m$ bits of
$\tau$ (for $m\ge m_0$), and the positions of these bits can be
reconstructed if we know $k \bmod 2^m$ and the function $f$.
This additional information uses $O(m)$ bits (recall that
$m_0\le m$ and it determines $f$). So
        $$
\K(\omega([k,k+2^m)))\ge\K(\tau([0,a_m 2^m)))-O(m)\ge a_m2^m-O(m).
        $$
The last term $O(m)$ can be eliminated: increasing $a_m$ by
$O(m)/2^m$, and even more, say, by $m^2/2^m$, we do not affect
the convergence. (The bounds presented are literally true for prefix
complexity; plain complexity of prefixes of $\tau$ is a bit
smaller but the difference again can be easily absorbed by
a constant factor that
does not affect the convergence.)
\end{proof}

\section{Proof of Theorem~\ref{th3}}

The proofs of Theorem~\ref{th1} and Theorem~\ref{th3} are based on the same idea.
We start with proving Theorem~\ref{th3} as it is simpler.

\begin{proof}[Proof of Theorem 3]
Fix some probabilistic algorithm $A$. We need to prove that some property
(``there exists function $g$ such that $K(\omega([k,k+n))|k,n)\ge g(n)$
for all $k$ and $n$'') has probability $0$ with
respect to the output distibution of $A$.
Since there are countable many
computable functions $g$, it is enough to show that
\emph{for a given $g$} this
happens with probability $0$. So we assume that both $A$
(probabilistic algorithm) and $g$ (a
computable monotone unbounded function) are fixed,
and for a given $\varepsilon>0$ prove that the property
``$K(\omega([k,k+n))|k,n)\ge g(n)$ for all $k$ and $n$'' has probability
smaller than $\varepsilon$.

Assume first that probabilistic algorithm $A$ produce an
infinite output sequence with probability $1$, and therefore
defines a computable
probability distribution $P_A$ on the Cantor space of
infinite sequences.

Consider some $n$. First we prove that for large enough $N$
it is possible to select one ``forbidden'' string
of length $n$ for each starting position
$k=0,1,\ldots, N-1$ in such a way
that the event ``output sequence avoids
all the forbidden strings'' (at the corresponding
positions) has probability less than~$\varepsilon$.

It can be shows in different ways. For example,
we can use the following probabilistic argument.
Let us choose the forbidden strings randomly (independently
with the random bits used by $A$). For every
output sequnce of $A$ the probability that it avoids
all randomnly selected ``forbidden'' strings is $(1-2^{-n})^N$ which
is less than $\varepsilon$ if $N$ is sufficiently large.
Therefore, the overall probability of the event ``output of $A$
avoids all forbidden strings'' (with respect to the product
distribution) is less than $\varepsilon$. Now we use averaging in
different order and conclude that there exists one sequence of
$N$ forbidden strings with the required property.

After the existence of such a sequence is proved, it can be found
by exhaustive search (recall that $P_A$ is computable). Let us
agree that we use the first sequence with this property
(in some search
order) and estimate the complexity of forbidden strings
when length $n$ and position $k$ are known. The value of $N$
is a simple function of $n$ and $\varepsilon$ (which is fixed
for now, as well as $A$), and we do not need any other information
to construct forbidden strings. So their conditional complexity
is bounded and is less than $g(n)$ for large enough $n$.
So the probability that all the substrings in the output of $A$
will have complexity greater than $g(\text{their length})$, is
less than $\varepsilon$.

It remains to explain how to modify this argument for a
general case, withouth the assumption that $A$ generates
infinite sequences with probability $1$. Let us modify the function
$N(\varepsilon,n)$ in such a way that
$(1-2^{-n})^{N(n,\varepsilon)}<\varepsilon/2$.
Consider the probability of the event ``$A$
generates a sequence of length $N(n,\varepsilon)+n$''.
If somebody gives us (in addition to $n$ and $\varepsilon$)
an approximation from below
for this probability with error at most $\varepsilon/2$,
we may enumerate $A$'s output distribution on strings of
length $N+n$ and stop when the lower bound is reached.
Then we apply the argument above using this restricted
distribution and show that for this restricted distribution
the probability to avoid simple strings is less than $\varepsilon/2$,
which gives $\varepsilon$-bound for the full distribution
(since they differ at most by $\varepsilon/2$). It is important
here that the missing information is of size $\log(1/\varepsilon)+O(1)$,
so for a fixed $\varepsilon$ we need $O(1)$ additional bits.
\end{proof}

\section{Proof of Theorem~\ref{th1}}

The proof of Theorem~\ref{th1} is similar to the preceding one,
but more technically involved. In the previous argument we
were allowed to choose different forbidden strings for
different positions, and it was enough to use one
forbidden string for each position. Now we use the
same set of forbidden strings for all positions,
and the simple bound $(1-2^{-n})^N$ is replaced by
the following lemma.

\begin{lemma}
	\label{main-lemma}
Let $\alpha\in (0,1)$. For every $\varepsilon>0$ there
exist natural numbers $n$ and $N$
\textup(with $n<N$\textup) and random variables
$\mathcal{A}_{n}, \mathcal{A}_{n+1},\ldots, \mathcal{A}_N$
whose values are
subsets of $\mathbb{B}^n, \mathbb{B}^{n+1},\ldots,
\mathbb{B}^N$ respectively, that have the following properties:

\textup{(1)} the size of subset $\mathcal{A}_i$ never exceeds $2^{\alpha i}$;

\textup{(2)}~for every binary string $x$ of length $N$ the
probability of the event ``for some $i\in \{n,\ldots,N\}$
some element of $\mathcal{A}_i$ is a substring of $x$''
exceeds $1-\varepsilon$.

The number $n$ can be chosen arbitrarily large.
\end{lemma}

(We again use the probabilistic argument; this lemma estimates
the probability for every specific $x$ and some auxiliary
probability distribution; the output distribution of
randomized algorithm $A$ is not mentioned at all. Then we
use this lemma to get an estimate for the combined distribution,
and change the order of averaging to prove the existence of
finite sets $A_n,\ldots,A_N$ with required properties.)

\begin{proof}
First let us consider the case $\alpha>1/2$.
Then we actually need only two lengths $n$ and $N$,
where $N\gg n$, all other lengths are not used
and the corresponding random subsets
can be empty. For length $n$, we consider a uniform
distribution on all
sets of size $2^{\alpha n}$; all these
sets have equal probabilities to be a value of
random variable $\mathcal{A}_n$. For length $N$
the set $\mathcal{A}_N$ is chosen in some fixed way
(no randomness), see below.

Assume that some string $x$ of length $N$ is fixed.
There are two possibilities:

\textup{(\textbf{a})}~there are at least $2^{n/2}$
different substrings of length $n$ in $x$;

\textup{(\textbf{b})}~there are less than $2^{n/2}$
different substrings.

In the first case (a) strings of length $n$ play the
main role. Let $S$ be a set of $n$-bit strings that
appear in $x$; it contains at least $2^{n/2}$ strings.
The probability that the desired event does not happen
does not exceed the probability of the following event:
``making $2^{\alpha n}$ random choices among $n$-bit
strings, we never get into $S$''. (It is a bit smaller,
since now we can choose the same string several times.)
This probability equals
	$$
(1-2^{n/2})^{2^{\alpha n}}=
(1-2^{n/2})^{2^{n/2}2^{(\alpha-1/2)n}}\approx
(1/e)^{2^{(\alpha-1/2)n}}
	$$
and converges to zero (rather fast) as $n\to\infty$.

In the second case (b) strings of length $N$ come into play.
We may assume that $N$ is a multiple of $n$.
Let us split $x$ into blocks of size $n$.
We know that $x$ has some special property:
there are at most $2^{n/2}$ different blocks.
Note that for large $N$ the number of strings
with this special property is less than $2^{\alpha n}$.
Indeed, to encode such a string $x$, we first list
all the blocks that appear in $x$
(this is a very long list, but its length is
determined by $n$ and does not depend on $N$),
and then specify each block by its number in this
list. In this way we need $N/2+O(1)$ bits
(the number is twice shorter than the block
itself) and this is less than
$\alpha N$ for large $N$. So for such a
large $N$ we may include all strings with
this property in $A_N$ and get the desired
effect with probability $1$.

Now let us consider the case when $\alpha>1/3$
(but can be less than $1/2$). Now we need three
lengths $n_1\ll n_2\ll n_3$. We will use $n_2$
that is a multiple of $n_1$, and $n_3$ that is a
multiple of $n_2$. For length $n_1$ we again
consider a random set of $2^{\alpha n_1}$
strings of length $n_1$. It guarantees success
if the string $x$ contains at least $2^{(2/3)n_1}$
different blocks of length $n_1$.

Now we compile a list of possible blocks
of size $n_2$ that are ``simple'', i.e.,
contain at most $2^{(2/3)n_1}$ different
blocks of size $n_1$. The same argument
as before shows that a simple block
can be described by $(2/3)n_2+O(1)$ bits,
where $O(1)$ depends only on $n_1$. Now $A_{n_2}$
is a random set of $2^{\alpha n_2}$
\emph{simple} blocks of size $n_2$.
Then the argument again splits into two subcases.
(Recall that we assume now that $x$ is made of
simple blocks of size $n_2$.)

The first case happens when $x$ contains more
than $2^{n_2/3}$ different simple blocks. Then with high
probability some block of $x$ appears in $A_{n_2}$.

The second case happen when $x$ contains less than $2^{n_2/3}$
different simple blocks. Then $x$ can be encoded by the list
of these blocks, and this requires $n_3/3+O(1)$ bits. So
if $n_3$ is large enough (compared to $n_2$), all possibilities
can be included in $A_{n_3}$, and this finishes the argument
for $\alpha>1/3$.

Similar argument with four layers works for $\alpha>1/4$, etc.
\end{proof}

This lemma will be the main technical tool in the proof
of Theorem~\ref{th1}. But first let us prove a purely
probabilitic counterpart of Theorem~\ref{th1} that is
of independent interest.

\begin{theorem}
Let $\alpha\in (0,1)$.
For every probability distribution $P$ on Cantor space
$\Omega$, there exist sets $A_1, A_2,\ldots$ of binary strings
such that

\textup{(1)} the set $A_n$ contains at most $2^{\alpha n}$
strings of length $n$;

\textup{(2)} with $P$-probability $1$ a random sequence has
substrings in $A_i$ for infinitely many $i$.

\end{theorem}

The possible ``philosophical'' interpretation
of this theorem: one cannot prove the existence
of sequences that avoid almost all $A_i$ by a direct
application of the probabilistic method, something
more delicate (e.g., Lovasz local lemma) is needed.

\begin{proof}
Let us first consider sequences of some finite length $N$
and the induced probability distribution on them. We claim
that for every $\varepsilon$ and for large enough $N$ we
can choose $A_1,\ldots,A_N$ in such a way that they satisfy
(1) and $P$-probability to avoid them is less than $\varepsilon$.

To show this, consider (independent) random distribution
on strings of lengths $1,\ldots,N$ provided by the lemma.
What is the probability that random string avoids random
set (with respect to the product distribution of $P$
and the distribution provided by the lemma)? Since for every
fixed string the probability
is less than $\varepsilon$ (assuming $N$ is large enough),
the overall probability (the average) is less than $\varepsilon$.
Changing the order of averaging, we see that for some
$A_1,\ldots,A_N$ the corresponding $P$-probability is
less than $\varepsilon$.

Note that in fact we do not need short strings;
strings longer than any given $n$ are enough
(if $N$ is large). So we can use this argument
repeatedly with non-overlapping intervals
$(n_i,N_i)$ and $\varepsilon_i$ decreasing fast
(e.g., $\varepsilon_i=2^{-i}$). Then for $P$-almost
every sequence we get infinitely many violations.
Moreover, since the series $\sum \varepsilon_i$ is
converging, $P$-almost every sequence hits
all but finitely many of $A_i$
(Borel--Cantelli lemma).
\end{proof}

Now we are ready to prove the weak version of Theorem~\ref{th1}:
\begin{quote}
\emph{Let $\alpha\in(0,1)$. There is no randomized algorithm
that produces $\alpha$-everywhere
complex sequences with probability $1$.}
\end{quote}

(The difference with the full version is that here we
have probability $1$ instead of any positive probability and
that the value of $\alpha$ is fixed.)

To prove this statement, let us
consider the output distribution $P$ of this
algorithm. Since the algorithm produces an infinite
sequence with probability~$1$, this distribution is a
computable probability distribution on the Cantor space.
This measure can be then used to effectively find
sequences $\varepsilon_i$, $n_i$, $N_i$ and sets $A_i$ as described
so that with $P$-probability $1$ a random sequence hits all
$A_i$ except for finitely many of them.
Since the sets $A_i$ can be effectively
computed and have at most $2^{\alpha i}$ elements,
every element of $A_i$ has complexity at
most $\alpha i+O(\log i)$; the logarithmic term can be absorbed
by a change in~$\alpha$.

This argument shows also that for every computable probability
distribution $P$ and every $\alpha\in(0,1)$
there exists a Martin-L\"of random
sequence with respect to $P$ that is not $\alpha$-everywhere complex.
One more corollary: for every $\alpha\in(0,1)$ the
(Medvedev-style) mass problem
``produce an $\alpha$-everywhere complex sequence'' is not
Medvedev (uniformly) reducible to the problem
``produce a Martin-L\"of random sequence''.

It remains to make the last step to get the proof of Theorem~\ref{th1}.

\begin{proof}[Proof of Theorem~\ref{th1}]

If the probability to get an sequence that is $\alpha$-everywhere
complex is positive, then for some $\alpha$ the probability
to get a $\alpha$-everywhere complex sequence for this specific
$\alpha$ is positive. (Indeed, we may consider only rational $\alpha$
and use countable additivitey.)

So we assume that some $\alpha$ is fixed and some probabilistic
algorithm produces $\alpha$-everywhere complex sequences with positive
probability. We cannot apply the same argument as above.
The problem is that the output of the algorithm (restricted
to the first $N$ bits) is a distribution on $\mathbb{B}^N$
that is not computable (the probability that at
least $N$ bits appear at the output, is only a
lower semicomputable real). However, for applying
our construction for some $\varepsilon_i$, it is
enough to know the output distribution up to
precision $\varepsilon_i/2$ (in terms of statistical distance),
as explained in the proof of Theorem~\ref{th3}, we replace
our distribution by its part, and the error is at most $\varepsilon/2$.
For this we need only $\log (1/\varepsilon_i)+O(1)$ bits of
advice, which can be made small compared to $\alpha n$.
\end{proof}

Now we get a stronger statements for mass problems:

\begin{theorem}
The mass problem ``produce an everywhere complex sequence''
is not Muchnik \textup(non-uniformly\textup) reducible to the
problem ``produce a Martin-L\"of random sequence''.
\end{theorem}

\begin{proof}
Indeed, imagine that for every random sequence there is some
oracle machine that transforms it to an everywhere complex
sequence. Since the set of oracle machines is countable, some
of then should work for a set of random sequences that has
positive measure, which contradicts Theorem~\ref{th1}.
\end{proof}

The author thanks Steven Simpson for asking the question,
and Joseph Miller and Mushfeq Khan for the discussion and useful remarks.

\end{document}